\documentclass[a4paper,11pt,pdf]{amsart}
\usepackage{enumerate, amsmath, amsfonts, amssymb, amsthm, thmtools, wasysym, graphics, graphicx, xcolor, frcursive,comment,bbm}
\usepackage[colorlinks=true,citecolor=cyan, backref=page]{hyperref}

\usepackage{etex}
\usepackage{lscape}
\usepackage{tikz-cd}

\makeatletter
\newcommand{\thickhline}{%
	\noalign {\ifnum 0=`}\fi \hrule height 1pt
	\futurelet \reserved@a \@xhline
}

\definecolor{darkblue}{rgb}{0.0,0,0.7} 
 
\definecolor{darkred}{rgb}{0.7,0,0} 
\usepackage{hyperref}
\usepackage[all]{xy}
\usepackage[T1]{fontenc}
\usepackage{adjustbox}

\usepackage{vmargin}            
\setmarginsrb{3cm}{2.5cm}{3cm}{2.5cm}{0cm}{0.6cm}{0cm}{0cm}

\usepackage{caption,lipsum}
\usepackage{graphicx}                  
\usepackage{pstricks,pst-plot,pst-text,pst-tree,pst-eps,pst-fill,pst-node,pst-math}
\usepackage{setspace}
\usepackage{multicol}

\newcommand{\darkred}{\color{darkred}} 
\newcommand{\defn}[1]{\emph{\darkred #1}}

\usepackage{etex}

\newtheorem{theorem}{Theorem}[section]
\newtheorem{prop}[theorem]{Proposition}
\newtheorem{lemma}[theorem]{Lemma}
\newtheorem{cor}[theorem]{Corollary}

\theoremstyle{definition}
\newtheorem{definition}[theorem]{Definition}
\newtheorem{rmq}[theorem]{Remark}

\numberwithin{equation}{section}

\title{Absolute order and involutions}

\author{Thomas Gobet}
\address{Université Clermont Auvergne, LMBP, UMR 6620 (CNRS), Campus des Cézeaux, 3 place Vasarely, TSA 60026, CS 60026, 63178 Aubière cedex, France}

\begin{document}
	\maketitle
	
	\begin{abstract}
		We study the restriction of the absolute order on a Coxeter group $W$ to an interval $[1,w]_T$, where $w\in W$ is an involution. We characterize and classify those involutions $w$ for which $[1,w]_T$ is a lattice, using the notion of \textit{involutive parabolic subgroups}.  
	\end{abstract}
	
	\tableofcontents

	\section{Introduction}
	
Let $(W,S)$ be a Coxeter system with set $T = \bigcup_{w\in W} w S w^{-1}$ of reflections. Let $\ell_T : W \longrightarrow \mathbb{Z}_{\geq 0}$ denote the length function on $W$ with respect to the set $T$ of generators, called \defn{reflection length}. The \defn{absolute order} on $W$, denoted $\leq_T$, is defined by $$ u \leq_T v \ \Leftrightarrow \ \ell_T(u) + \ell_T(u^{-1} v) = \ell_T(v).$$	
The reflection length was first considered by Carter~\cite{Carter} to classify conjugacy classes in the Weyl groups. Both the reflection length and absolute order play a crucial role in the construction of dual Artin groups attached to finite Coxeter groups. In this context, denoting $[u,v]_T :=\{ w\in W \ \vert \ u \leq_T w \leq_T v\}$, such groups are built as so-called \textit{interval groups} from the intervals $[1,c]_T$, where $c$ is a Coxeter element in $W$, that is, a product of the elements of $S$ in some order. An important feature in that, when $W$ is finite, the poset $[1,c]_T$ is a lattice, isomorphic if $W$ is the symmetric group to the lattice of noncrossing partitions~\cite{Biane}. The absolute order therefore also allows one to generalize the lattice of noncrossing partitions to every finite Coxeter group, and is thus also actively studied in connection with combinatorial structures related to noncrossing partitions--see for instance~\cite{Arm, STW}. In the interval group setting, the lattice property ensures that the attached interval group is a so-called \textit{Garside group}. Such a group is torsion-free, admits a solution to the word and conjugacy problems, has a finite $K(\pi, 1)$, a nontrivial computable center, etc. In the case of a finite Coxeter group $W$, the interval group attached to $[1,c]_T$ turns out to be isomorphic to the Artin group of the same type. This motivates the study of the poset $[1,c]_T$, also in non-spherical type, as a tool to understand Artin groups of non-spherical type. For more on the topic we refer to~\cite{Bessis, Paolini}. For which $W$ and which choice of Coxeter element $c$ the poset $[1,c]_T$ is a lattice remains a difficult, open question in the field. A solution to this question is known only for finite, affine, universal, and rank three Coxeter systems. 

In fact, every element $w\in W$ such that $[1,w]_T$ is a lattice gives rise to an interval Garside group. It is thus natural to seek for elements $w$ satisfying this property. We summarize the known results on this question: 
\begin{itemize}
	\item If $W$ is finite~\cite{Bessis, BW, Reading}, universal~\cite{Bessis_free}, or of rank three~\cite{DPS, Gobet}, and $c$ is a (parabolic) Coxeter element, then $[1,c]_T$ is a lattice,
	\item If $W$ is of type $\widetilde{A}_n$ or $\widetilde{C}_n$~\cite{Dig1, Dig2}, then for suitable choices of Coxeter elements $[1,c]_T$ is a lattice, and for other choices it is not, 
	\item If $W$ is of affine type and not covered by the two previous points, then for any choice of Coxeter element, $[1,c]_T$ is not a lattice~\cite{McCammond}, 
	\item In type $A_n$, $[1,w]_T$ is a lattice for every $w\in W$ (as every element is a parabolic Coxeter element in this case),
	\item In types $B_n$ and $D_n$, there is a full classification of elements $w$ such that $[1,w]_T$ is a lattice~\cite{Kalli} (it is never the whole group), 
	\item If $W$ is arbitrary and $w\in W$ is such that $\ell_T(w) \leq 3$, then $[1,w]_T$ is a lattice~\cite{Gobet}.
\end{itemize}  
	
The aim of this paper is to classify the involutions $w$ such that $[1,w]_T$ is a lattice. We develop uniform tools to answer the above question. More precisely, given an involution $w$ in a Coxeter group $W$, denote by $P(w)$ its parabolic closure, which turns out to be finite. Since $P(w)$ is a reflection subgroup of $W$, it is again a Coxeter group, with set of generators $T_w$. The poset $[1,w]_T$ fully lies inside $P(w)$, and coincides with $[1,w]_{T_w}$. We can thus restrict our study to finite Coxeter groups which are parabolic closures of involutions. We then show in Proposition~\ref{prop_poset_isom} that $[1,w]_T$ coincides with the set of involutions in $P(w)$, and in Theorem~\ref{thm_lattice} that $[1,w]_T$ is a lattice if and only if the set of parabolic closures of involutions in $P(w)$ is stable by intersection. Finally, using these results, we classify in Section~\ref{sec_class} the involutions $w$ such that $[1,w]_T$ is a lattice. 

In a forthcoming joint paper with E.~Chavli, we will identify and study the attached interval Garside groups. 

\medskip 

\textbf{Acknowledgements.} I thank Eirini Chavli for useful discussions.

	\section{Preliminaries}
	
	Let $(W,S)$ be a Coxeter system. Let $V$ be its geometric representation and $T:=\bigcup_{w\in W} w S w^{-1}$ its set of reflections (for basics on Coxeter groups we refer the reader to~\cite{Bou, Humphreys}). We denote by $\Phi = \Phi^+ \sqcup \Phi^-$ the corresponding root system and by $t \leftrightarrow \alpha_t$ the $1$ to $1$ correspondence between reflections and positive roots. Recall that there is a function $N: W \longrightarrow \mathcal{P}(T)$ such that $N(s) = \{s\}$ for all $s\in S$ and $N(xy) = N(x) \ \Delta \ x N(y) x^{-1}$, where $\Delta$ denotes symmetric difference. One has $|N(w)|=\ell_S(w)$ for all $w\in W$, where $\ell_S$ stands for the length function on $W$ with respect to the generating set $S$. The elements of $N(w)$ are called the \defn{(left) inversions} of $w$ and may be characterized geometrically as those $t\in T$ such that $w^{-1} (\alpha_t) \in \Phi^{-}$, equivalently such that $\ell_S(tw) < \ell_S(w)$. If $W$ is finite we say that $(W,S)$ is a finite Coxeter system, and denote by $w_0$ its longest element, which is involutive and satisfies $N(w_0)=T$.
	
	The \defn{reflection length} $\ell_T(w)$ of an element $w\in W$ is the smallest number $k$ for which there exist reflections $t_1, t_2, \dots, t_k\in T$ such that $w=t_1 t_2 \cdots t_k$. If $t_1, t_2, \dots, t_k\in T$ and $w= t_1 t_2 \cdots t_k$ with $k=\ell_T(w)$, we say that the word $t_1 t_2 \cdots t_k$ is a \defn{$T$-reduced expression} (versus a $\defn{$S$-reduced expression}$ when $t_i\in S$ and $k=\ell_S(w)$) of $w\in W$. The \defn{absolute order} $\leq_T$ on $W$ is then defined as $ u \leq_T v$ if and only if $\ell_T(u) + \ell_T(u^{-1} v)= \ell_T(v)$, that is, there is a $T$-reduced expression of $v$ having a $T$-reduced expression of $u$ as a prefix (equivalently as a suffix, since $T$ is stable by conjugation). It is natural to seek for geometric interpretations of $\ell_T(w)$ and algorithms to determine it. When $W$ is finite we have the following:
	
	\begin{theorem}[Carter's Lemma, {\cite[Lemmas 1 - 3]{Carter}}]\label{thm_carter}
		Let $(W,S)$ be a finite Coxeter group. For $w\in W$, let $V^w=\{v\in V \mid w(v)=v\}$, which we rather write $H_t$ if $w=t\in T$. Then 
		\begin{enumerate}
			\item $\ell_T(w)= \mathrm{dim}(V) - \mathrm{dim}(V^w)$, for all $w\in W$. 
			\item For all $t\in T$, $w\in W$, we have $t \leq_T w \Leftrightarrow V^w \subseteq H_t.$
			\item Given $w\in W$ and $t_1, t_2, \dots, t_k\in T$ such that $w= t_1 t_2 \cdots t_k$, we have $\ell_T(w)=k$ if and only if the roots $\{\alpha_{t_i}\}_{i=1}^k$ are linearly independent. 
		\end{enumerate}
	\end{theorem}
	
	Let $\mathrm{Mov}(w):= \mathrm{Im}(w - \mathrm{Id}_V)$. It is called the \defn{moved space} of $w$ and is stable by $w$. We have $\mathrm{Mov}(w)^\perp = V^w$ (see~\cite[Proposition 1]{BW_o}). In particular $\ell_T(w)=\dim(\mathrm{Mov}(w))$. 
	
	Theorem~\ref{thm_carter} does not generalize to arbitrary $W$, where typically the reflection length can be unbounded in general~\cite{Dus}. In affine types one can derive a formula analogous to the one in point $1$ (see~\cite{BLMC}). 
	
	For arbitrary Coxeter systems, we have the following: 
	
	\begin{theorem}[Dyer, {\cite[Theorem 1.1]{Dyer_refn}}]\label{thm_dyer_ref_length} 
		Let $(W,S)$ be an arbitrary Coxeter system. Let $w\in W$. Let $s_1 s_2 \cdots s_k$ be any $S$-reduced expression of $w$. Then $\ell_T(w)$ is equal to the minimal number of letters to delete in the word $s_1 s_2 \cdots s_k$ to get a word representing the identity.  
	\end{theorem}
	
Let $I\subseteq S$. Recall that the subgroup $(W_I:=\langle I \rangle, I)$ is itself a Coxeter group, called \defn{standard parabolic subgroup}. A \defn{parabolic subgroup} is a subgroup of the form $w W_I w^{-1}$ for some $w\in W$. Parabolic subgroups are examples of reflection subgroups, that is, subgroups generated by a subset of $T$, which themselves carry a canonical structure of Coxeter group~\cite{Dyer_subgroups, Deodhar}--recall that if $W'\subseteq W$ is a reflection subgroup of $W$, then its set $T'$ of reflections with respect to its canonical Coxeter group structure is given by $W'\cap T$. Parabolic subgroups can also be described as centralizers of subspaces of $V$.  

Parabolic subgroups are stable by intersection, hence given $w\in W$ one can consider its \defn{parabolic closure} $P(w)$, defined as the smallest parabolic subgroup containing $w$. 

\medskip

The following is folkloric (see for instance~\cite[Section 2.4]{Gobet_dual}), and will be useful later on. 

\begin{prop}\label{prop_abs}
Assume that $W$ is finite. Let $w, x\in W$, $t\in T$. Then 
\begin{enumerate}
	\item We have $P(w)= C_W(V^w):=\{w\in W \ \vert \ w(v)=v \ \forall v\in V^w\}$, and $\mathrm{rank}(P(w))=\ell_T(w)$,
	\item If $w \leq_T x$, then $V^x \subseteq V^w$, hence $P(w) \subseteq P(x)$,
	\item  We have $t \leq_T w \Leftrightarrow t \in P(w)$. 
\end{enumerate}
\end{prop}	
Since $P(w)$ is a reflection subgroup of $W$, hence has set of reflections $P(w) \cap T$, we have that the set of elements of $T$ such that $t\leq_T w$ is precisely the set of reflections of $P(w)$, hence it is in particular a generating set. 

	\section{Involutive parabolic subgroups and lattices of involutions}
	
	Let $(W,S)$ be a Coxeter system. 
	
	\begin{definition}
		Let $P\subseteq W$ be a parabolic subgroup. We say that $P$ is \defn{involutive} if there is an element $x\in W$ such that $x^2=1$ and $P(x)=P$. 
	\end{definition}
	
	\begin{lemma}\label{p_finite}
	If $P$ is involutive, then $P$ is finite. 
	\end{lemma}
	
	\begin{proof}
	Let $x\in W$ such that $x^2=1$ and $P=P(x)$. In particular $x$ has finite order. By~\cite[Proposition 1.3]{BH}, there is $I \subseteq S$ and $w\in W$ such that $W_I$ is finite and $x\in w W_I w^{-1}$. Since $w W_I w^{-1}$ is a parabolic subgroup containing $x$, we have $P \subseteq w W_I w^{-1}$. The latter being finite, we deduce that $P$ is finite. 
	\end{proof}
	
The above lemma, together with Proposition~\ref{prop_abs}, will allow us to restrict to finite Coxeter groups when studying $[1,w]_T$ with $w$ an involution. 
	
	\begin{prop}\label{prop:invol}
		Let $(W,S)$ be a finite Coxeter system, and let $u\in W$ be an involution. Let $t\in T$. Then $$t\leq_T u \ \Leftrightarrow \ tu=ut\text{~and~} t\in N(u).$$ 
	\end{prop}
	
	\begin{proof}
		Assume that $t\leq_T u$. Then $V^u \subseteq H_t$. We thus have $\mathbb{R} \alpha_t =(H_t)^\perp \subseteq (V^u)^\perp$. But since $u$ is involutive, we have that $(V^u)^\perp$ is precisely the $(-1)$-eigenspace of $u$. We thus have $u(\alpha_t)=-\alpha_t$, hence $t\in N(u)$. Moreover, since $\mathbb{R} \alpha_t$ is stable by $u$, we have that $H_t=(\mathbb{R} \alpha_t)^\perp$ is also stable by $u$. Since $\mathbb{R} \alpha_t$ and $H_t$ are the eigenspaces of $t$, it follows that $tu=ut$.
		
		Conversely, assume that $tu=ut$ and that $t\in N(u)$. Since $t$ and $u$ commute with each other, they are simultaneously diagonalizable. It follows that $u$ preserves $\mathbb{R} \alpha_t$ and $H_t$, which are the eigenspaces of $t$. It follows that $u(\alpha_t)\in\{\pm \alpha_t\}$. Since $t\in N(u)$, this forces $u(\alpha_t)=-\alpha_t$. Writing $H_t= E_1 \oplus E_{-1}$, where $u$ acts by $\mathrm{Id}_V$ on $E_1$ and by $-1$ on $E_{-1}$, we get that $$V= E_1 \oplus E_{-1} \oplus \mathbb{R} \alpha_t,$$ where $u t$ acts by $\mathrm{Id}_V$ on $E_1 \oplus \mathbb{R} \alpha_t$ and by $-\mathrm{Id}_V$ on $E_{-1}$. Now $V^{u}=E_1$ while $V^{u t}=E_1 \oplus \mathbb{R} \alpha_t$, hence $\ell_T(u)=\ell_T(u t)+1$, which concludes the proof.  
	\end{proof}
	
	\begin{cor}\label{cor_wowo}
		Let $(W,S)$ be a finite Coxeter system and $t\in T$. Then $$ t \leq_T w_0 \ \Leftrightarrow \ w_0 t= t w_0.$$  
	\end{cor}
	
	\begin{proof}
	If follows immediately from the previous proposition together with the fact that $N(w_0)=T$. 	
		\end{proof}

	\begin{cor}\label{cor_wo}
		Let $(W,S)$ be a finite Coxeter system and $u\in W$. Then $u^2=1$ if and only if $u$ is the longest element in its parabolic closure $P(u)$. In particular, if $P$ is an involutive parabolic subgroup, then there is a unique involution $u\in P$ such that $P=P(u)$. 
	\end{cor}
	
	\begin{proof}
		Assume that $u$ is an involution. Let $t\in P(u)$. By point $3$ of Proposition~\ref{prop_abs} we have $t\leq_T u$. By Proposition~\ref{prop:invol}, it implies that $t\in N(u)$. It follows that every reflection in $P(u)$ is an inversion of $u$, which happens if and only if $u$ is the longest element of $P(u)$. The converse is immediate as longest elements of finite Coxeter groups are always involutive.  
	\end{proof}

	\begin{prop}\label{prop_invol_bis}
		Let $(W,S)$ be a finite Coxeter system and $u\in W$ be an involution. Let $v\in W$ such that $ v\leq_T u$. Then $v$ is an involution, and $vu=uv$. 
	\end{prop}
	
	\begin{proof}
		We begin by showing the second statement. Since $v \leq_T u$, we have $V^u \subseteq V^v$. Hence $v$ preserves $V^u$, and thus it also preserves $(V^u)^\perp=\mathrm{Mov}(u)$. Since $V^u$ and $\mathrm{Mov}(u)$ are the eigenspaces of $u$, it follows that $uv=vu$. 
		
		For the first statement, we begin by observing that \begin{equation}\label{eq_dir_s} V= \mathrm{Mov}(u)\oplus V^u= \mathrm{Mov}(v)\oplus \mathrm{Mov}(v^{-1}u) \oplus V^u,\end{equation} where the second equality holds by~\cite[Corollary 1]{BW_o}. Since $u$ is involutive, it acts by $-\mathrm{Id}_V$ on $(V^u)^\perp=\mathrm{Mov}(u)=\mathrm{Mov}(v)\oplus \mathrm{Mov}(v^{-1}u)$. It suffices to show that $v$ acts by $-\mathrm{Id}_V$ on $\mathrm{Mov}(v)$ to conclude that $v$ is involutive. Hence let $x\in \mathrm{Mov}(v)$. We have $$v^{-1}(v(x)+x)= v^{-1}(v(x)-u(x))=-(v^{-1}u(x)-x),$$ but $v^{-1}(v(x)+x)$ lies in $\mathrm{Mov}(v)$ since $v(\mathrm{Mov}(v))=\mathrm{Mov}(v)$, while $-(v^{-1}u(x)-x)\in\mathrm{Mov}(v^{-1}u)$. The decomposition~\eqref{eq_dir_s} thus forces $v^{-1}(v(x)+x)=0$, hence $v(x)=-x$. 
	\end{proof}
	
	In~\cite[Propositions 2.4 and 2.5]{Serre}, it is proven that every involution in a finite Coxeter group admits a $T$-reduced decomposition made of reflections which are pairwise commuting (such decompositions also appeared in the work of Deodhar~\cite{Deodhar_invol} and Springer~\cite{Sp}, without the notion of reflection length). We deduce from the above proposition that this result holds for every $T$-reduced decomposition:
	
	\begin{cor}
	Let $(W,S)$ be a finite Coxeter system and $u\in W$ an involution. Let $t_1 t_2 \cdots t_k$ be a $T$-reduced expression of $u$. Then the reflections $t_1, t_2, \cdots, t_k$ are pairwise commuting. 
	\end{cor}
	
	\begin{proof}
	Let $1 \leq i < j \leq k$. Since $u = t_i t_j {}^{t_j t_i}( t_1 \cdots t_{i-1}) {}^{t_j}(t_{i+1} \cdots t_{j-1}) (t_{j+1} \cdots t_k)$ and the conjugate of a reflection is again a reflection, we have $t_i t_j \leq_T u$. We deduce from the previous proposition that $t_i t_j$ is an involution, hence $t_i t_j = t_j t_i$.  
	\end{proof}
	
  \begin{rmq}
  A classical tool in the study of posets $[1,c]_T$ where $c$ is a (parabolic) Coxeter element is the so-called \textit{Hurwitz action} on $T$-reduced expressions of $c$~\cite[Section 1.6]{Bessis}. It is an action of the braid group on $\ell_T(c)$ strands on such expressions. A property which has many important consequences in the study of dual Artin groups is the transitivity of this action. It can be defined on $T$-reduced decompositions of an arbitrary $w\in W$, but in the case of an involution this tool is quite irrelevant: except in particular cases, it has no chance to be transitive since when you have a $T$-reduced decomposition made of pairwise commuting reflections, the action just modifies the order of the factors. In particular, there are in general several Hurwitz orbits, resulting for example in relations of the form $t_1 t_2 = t_3 t_4$ with $|\{t_1, t_2, t_3, t_4\}|=4$ occurring in the case of involutions, for instance in type $B_2$ with $w=w_0$ and $S=\{s,t\}$, we have $w_0 =t (sts) = s (tst)$. Such relations are not captured by the Hurwitz action. 
  \end{rmq}

	\begin{prop}\label{prop_invol_3}
		Let $(W,S)$ be a finite Coxeter system and let $u\in W$ be an involution. Let $P(u)$ denote the parabolic closure of $u$. Let $v\in W$. The following are equivalent: 
		\begin{enumerate} 
			\item $v \leq_T u$, 
			\item  $v^2=1$ and $v\in P(u),$ 
			\item $v^2=1$ and $V^u \subseteq V^v$. 
		\end{enumerate}	
			If any of the above conditions holds, we have $uv=vu$. 
	\end{prop}
	
	\begin{proof}
		Suppose that $(1)$ holds. We deduce $(2)$ and also the last statement using Proposition~\ref{prop_invol_bis} and point $2$ of Proposition~\ref{prop_abs}. Assume $(2)$. Since $v\in P(u)=C_W(V^u)$, we have $V^u \subseteq V^v$, which gives $(3)$. Assume $(3)$. We have $V^u \subseteq V^v$, hence $V^v = V^u \oplus (V^v \cap \mathrm{Mov}(u))$ yielding $$V= V^v \oplus \mathrm{Mov}(v)= (V^u \oplus (V^v \cap \mathrm{Mov}(u))) \oplus \mathrm{Mov}(v).$$ Since $u$ and $v$ are involutive and $\mathrm{Mov}(v)= (V^v)^\perp \subseteq (V^u)^\perp = \mathrm{Mov}(u)$, we have that $u$ (respectively $v$) acts by $1$ on $V^u$ (respectively by $1$), by $-1$ on $V^v \cap \mathrm{Mov}(u)$ (respectively by $1$), and by $-1$ on $\mathrm{Mov}(v)$  (respectively by $-1$). We thus have that $v^{-1}u$ is involutive with $V^{v^{-1}u} = V^u \oplus \mathrm{Mov}(v)$ and $\mathrm{Mov}(v^{-1}u)= V^v \cap \mathrm{Mov}(u)$. We deduce that \begin{align*}\ell_T(v)+ \ell_T(v^{-1} u)&=\dim(\mathrm{Mov}(v))+\dim(V^v \cap \mathrm{Mov}(u))= \dim(V) - \dim(V^u) \\&=\dim(\mathrm{Mov}(u))=\ell_T(u),\end{align*} which shows that $v\leq_T u$, yielding $(1)$.     
	\end{proof}

	\begin{prop}\label{prop_poset_isom}
	Let $(W,S)$ be a finite Coxeter system. Let $u\in W$ be an involution. Then $[1,u]_T$ is the set of involutions of $P(u)$. Moreover, denoting by $\mathcal{IP}(u)$ the set of involutive parabolic subgroups of $P(u)$, we have an isomorphism of posets $$\varphi: [1,u]_T \longrightarrow (\mathcal{IP}(u), \subseteq), \ v \mapsto P(v).$$
	\end{prop}
	
	\begin{proof}
	The first statement is an immediate consequence of the equivalence between points $1$ and $2$ of Proposition~\ref{prop_invol_3}.
	
	We now prove the second statement. Note that, since $P(u)$ is parabolic, the parabolic subgroups of $P(u)$ are exactly the parabolic subgroups of $W$ included in $P(u)$ (this is easily seen for instance using the description of parabolic subgroups as centralizers of subspaces in the geometric representation). Let $v_1, v_2\in [1,u]_T$, which by the already proven statement are involutions. If $\varphi(v_1) = \varphi(v_2)$ then $v_1 \in P(v_2)$ and $v_2 \in P(v_1)$. Using the equivalence between points $1$ and $2$ of Proposition~\ref{prop_invol_3} we obtain that $v_1 \leq_T v_2$ and $v_2 \leq_T v_1$, hence $v_1 = v_2$. The map is thus injective. It is surjective since, if $P$ is an involutive parabolic subgroup of $P(u)$, then there is an involution $v\in P(u)$ such that $P=P(v)$. By Proposition~\ref{prop_invol_3} again we deduce that $v \leq_T u$, hence $v\in [1,v]_T$. Finally, the map $\varphi$ is an isomorphism of posets since, using again Proposition~\ref{prop_invol_3}, we have $$ v \leq_T u \Leftrightarrow v \in P(u) \Leftrightarrow P(v) \subseteq P(u).$$
	\end{proof}

	\begin{theorem}\label{thm_lattice}
		Let $(W,S)$ be a finite Coxeter system. Let $u\in W$ such that $u^2=1$. Then $[1,u]_T$ is a lattice if and only if the set $\mathcal{IP}(u)$ of involutive parabolic subgroups of $P(u)$ is stable by intersection. In this case in the notation of Proposition~\ref{prop_poset_isom}, for $v, w\in [1,u]_T$ we have $\varphi( v \wedge w) = P(v) \cap P(w)$. 
	\end{theorem}
	
	\begin{proof}
		Assume that the involutive parabolic subgroups of $P(u)$ are stable by intersection. It suffices to show the existence of meets since the poset is finite and has a maximal element $P(u)$ (see~\cite[Proposition 3.3.1]{Stanley}). Hence let $v, w\in [1,u]_T$. Then by Proposition~\ref{prop_invol_3}, we have $v, w\in P(u)$ and $v^2=w^2=1$. In particular, the parabolic subgroups $P(v)$ and $P(w)$ are involutive, hence by assumption the same holds for $P=P(v)\cap P(w)$. There is thus $z\in P$ such that $P=P(z)$ and $z^2=1$. By Proposition~\ref{prop_invol_3}, we have $z \leq_T v, w$. Now let $z'\in W$ such that $z' \leq_T v$, $z' \leq_T w$. In particular we have $z'\in P(w)\cap P(y)= P(z)$. Since $z\in [1,u]_T$, by Proposition~\ref{prop_invol_bis} it is an involution. Using Proposition~\ref{prop_invol_3} we deduce that $z' \leq_T z$.

		Conversely, assume that the involutive parabolic subgroups of $P(u)$ are not stable by intersection. There are thus two involutions $v, w\in P(u)$ such that $P:=P(v) \cap P(w)$ is not involutive. Consider the longest element $z_1=w_{0,P}$ of $P$. Since it is an involution, by Proposition~\ref{prop_invol_3} we have $z_1 \leq_T v$ and $z_1 \leq_T w$. Since $P$ is not involutive, we have $P(z_1) \subsetneq P$. Since both $P(z_1)$ and $P$ are generated by reflections, there is thus $t\in P\cap T$ such that $t\notin P(z_1)$. We hence have $t \not\leq_T z_1$, and by Corollary~\ref{cor_wowo} applied to the Coxeter group $P$, it implies that $t z_1 \neq z_1 t$. Setting $z_2 := t z_1 t$, we have that $z_2$ is a involution lying in $P$ and distinct from $z_1$, with $\ell_T(z_1)=\ell_T(z_2)$. By Proposition~\ref{prop_invol_3} we thus have $z_2 \leq_T v, w$. We can now easily deduce that $v$ and $w$ cannot have a meet for $\leq_T$: assume for contradiction that $Z\in [1,u]_T$ is a meet of $v$ and $w$. We then have that $Z$ is an involution, that $Z\leq_T v$ and $Z\leq_T w$, hence $Z \in P(v)\cap P(w)=P$, and $z_1, z_2 \leq_T Z$. Since $\ell_T(z_1) = \ell_T(z_2)$ and $z_1 \neq z_2$, we have $z_1 <_T Z$. Hence there is $t\in P \cap T$ such that $z_1 \leq_T z_1 t \leq_T Z$. Since $Z$ is an involution, by Proposition~\ref{prop_invol_bis} we get that $z_1 t$ is also an involution. We thus have $z_1 t = t z_1$, hence $t \leq_T z_1$ by Corollary~\ref{cor_wowo}, contradicting $z_1 \leq_T z_1 t$.    
		
		The last statement follows immediately from the isomorphism of posets in Proposition~\ref{prop_poset_isom}, since the meet in the lattice of involutive parabolic subgroups is clearly given by the intersection.    
	\end{proof}
	
	Given an involution $u$ of a finite Coxeter group $W$, we thus obtain that $[1,u]_T$ is a lattice if and only if the involutive parabolic subgroups of $P(u)$ are stable by intersection; if $P(u)$ is not irreducible, then as we shall see in Lemma~\ref{dec_tech_pu} below $P(u)$ will decompose as a direct product $P(u_1) \times P(u_2) \times \cdots \times P(u_k)$, where each $u_i$ is an involution, and $u= u_1 u_2 \cdots u_k$. The lattice will then be given by the direct product of the lattices $[1,u_i]_T$. Classifying involutions $u$ such that $[1,u]_T$ is a lattice thus reduces to classifying finite irreducible and involutive Coxeter groups in which involutive parabolic subgroups are stable by intersection. The first step is thus to determine the finite irreducible and involutive Coxeter groups, which will be easy with the help of the next proposition. 
	
		\begin{prop}\label{prop_invol_p_car}
		Let $(W,S)$ be a finite, irreducible Coxeter system and let $u\in W$ be an involution. The following are equivalent:
		\begin{enumerate}
			\item $W=P(u)$,
			\item $t \leq_T u$ for all $t\in T$, 
			\item $u\in Z(W)$, 
			\item $\ell_T(u)= \mathrm{rank}(W)$.
		\end{enumerate}
		In this case we have $u=w_0$, and $[1,u]_T$ is the set of involutions of $W$. 
	\end{prop}
	
	\begin{proof}
		$1 \Leftrightarrow 2$. Since $T$ generates $W$, this equivalence is obtained using point $(3)$ of Proposition~\ref{prop_abs}. 
		
		$2 \Leftrightarrow 3$. Assume $2$. Hence $t \leq_T u$ for all $t\in T$. By Proposition~\ref{prop:invol}, we deduce that $tu=ut$ for all $t\in T$, hence that $u\in Z(W)$ since $T$ generates $W$. Conversely, assume that $u \in Z(W)$. Then for all $t\in T$, we have that $ut$ is an involution. Let $T= T_1 \cup T_2$, where $T_1 =\{t\in T \ \vert \ t \leq_T u\}$ and $T_2 =\{t\in T \ \vert \ t \not\leq_T u\}$. Since $u\neq 1$ we have $T_1 \neq \emptyset$. Assume that $T_2 \neq \emptyset$ and let $t_i \in T_i$, $i=1, 2$. Then $u t_1 \leq_T u \leq_T u t_2$, hence $u t_2$ is an involution such that $t_1 t_2 \leq_T u t_2$. By Proposition~\ref{prop_invol_bis} we deduce that $t_1 t_2$ is an involution, hence $t_1 t_2 = t_2 t_1$. The reflections in $T_1$ therefore commute with the reflections in $T_2$. We thus obtain a decomposition as a direct product $W= \langle T_1 \rangle \times \langle T_2 \rangle$, which is impossible since $W$ is irreducible. Hence $T_2 = \emptyset$ and $T_1=T$, yielding $2$.  
		
		The equivalence between $1$ and $4$ follows from the first point of Proposition~\ref{prop_abs}. 
	\end{proof}

	It is well-known that if the center of a finite, irreducible Coxeter group $W$ is nontrivial, then $Z(W)=\{1, w_0\}$ (see for instance~\cite[Chap. V, §4, Exercise 3]{Bou}). We thus get from the previous proposition that the finite, irreducible and involutive Coxeter groups coincide with those with a nontrivial center, or those for which the longest element has maximal reflection length, given by \begin{equation}\label{class_invol} A_1, I_2(2k), k \geq 2, B_n ,  n \geq 3, D_{2k}, k \geq 2, E_7, E_8, F_4, H_3, H_4.\end{equation}

	\begin{rmq}
	In~\cite{Serre}, several results on involutions in finite Coxeter groups are collected and proven; in particular, Serre defines the \defn{degree} of an involution $u$ as the codimension of $V^u$, hence the degree is exactly the reflection length. He calls an involution \defn{maximal} if it has maximal degree. Using different methods, Serre shows that maximal involutions are always conjugate~\cite[Corollaire 3.4]{Serre}, and that for every involution $v\in W$, there is a maximal involution $u\in W$ such that $ v \leq_T u$~\cite[Théorème 3.7]{Serre}. The setting of Proposition~\ref{prop_invol_p_car} corresponds to the case where there is a unique maximal involution, in which case it acts by $-\mathrm{Id}_V$ on $V$: indeed, let $u$ and $W$ be as in Proposition~\ref{prop_invol_p_car} and assume that any of the four equivalent conditions of the proposition holds. Since $\ell_T(u) = \mathrm{rank}(W)= \mathrm{dim}(V)$, we have $V = \mathrm{Mov}(u)$, and since $u$ is involutive, it follows that $u$ acts on $V$ as $-\mathrm{Id}_V$. In this case Serre's result also gives that for every involution $v\in W$, we have $v \leq_T u$.  
	\end{rmq}
	
\section{Classification of involutions yielding a lattice}\label{sec_class}

The aim of this section is to classify the involutions $w$ in a Coxeter group $W$ such that $[1,w]_T$ is a lattice. We begin by the following preliminary lemma, which allows us to reduce to the case where $W$ is finite, irreducible and involutive. 

\begin{lemma}\label{dec_tech_pu}
Let $(W,S)$ be a Coxeter system and $u\in W$ an involution. Then $P(u)$ is finite, and there are involutions $u_1, u_2, \dots, u_k\in W$, commuting to each other, unique up to reordering, such that $u = u_1 u_2 \cdots u_k$, $P(u_i)$ is irreducible for each $i$, and $$P(u) = P(u_1) \times P(u_2) \times \dots \times P(u_k).$$ The poset $[1,u]_T$ is then the direct product of the posets $[1, u_i]_T$. 
\end{lemma}

\begin{proof}
The fact that $P(u)$ is finite is given by Lemma~\ref{p_finite}. Since $P(u)$ is a parabolic subgroup, it is a Coxeter group, and we can decompose it into irreducible components $P_1, P_2, \dots, P_k$, which are themselves parabolic, and this decomposition is unique up to the order of the factors. We thus have $P(u) = P_1 \times P_2 \times \cdots P_k$, and $u$ decomposes uniquely according to this decomposition as $u = u_1 u_2 \cdots u_k$, where $u_i \in P_i$ and $u_i$ is involutive for all $i$. It remains to show that $P_i = P(u_i)$. If not, there is $i$ such that $P(u_i) \subsetneq P_i$. But then $u$ lies in $P_1 \times \cdots \times P_{i-1} \times P(u_i) \times P_{i+1} \times \cdots \times P_k$, which is strictly included in $P(u)$, and parabolic, a contradiction. We thus get the existence of the decomposition, and the uniqueness follows from the uniqueness of the decomposition of $P(u)$ into irreducible components and Corollary~\ref{cor_wo}.

The last statement is a consequence of the decomposition $P(u) = P(u_1) \times \cdots \times P(u_k)$ and of the fact that $\ell_T(u) = \sum_{i=1}^k \ell_T(u_i)$; this last statement holds since by the first point of Proposition~\ref{prop_abs} we have $$\ell_T(u) =\mathrm{rank}(P(u)) = \sum_{i=1}^k \mathrm{rank}(P(u_i)) = \sum_{i=1}^k \ell_T(u_i).$$
\end{proof}
	
We can now prove the main result, checking which finite irreducible Coxeter groups among those listed in~\eqref{class_invol} satisfy the condition of Theorem~\ref{thm_lattice}.	
	
\begin{theorem}\label{thm_main_lattice}
The finite, irreducible, involutive Coxeter groups $W$ such that $[1, w_0]_T$ is a lattice are those of type $A_1, I_2(2k), k\geq 2, B_n, n\geq 3, D_4, H_3$.
\end{theorem}

\begin{proof}
We may go through the list given in~\eqref{class_invol} and check for each family whether involutive parabolic subgroups are stable by intersection or not; applying Theorem~\ref{thm_lattice} then allows one to conclude. 

We first note that the case $B_n, n\geq 3$ and $D_{2k}, k\geq 2$ has already been threated by Kallipoliti in~\cite[Theorems 7 and 8]{Kalli}, who characterized all the elements $w$ in types $B_n$ and $D_n$ such that $[1,w]_T$ is a lattice. We can also reobtain it using the framework developed in this paper by checking that involutive parabolic subgroups are stable by intersection in these cases: in type $B_n$, using the list in~\eqref{class_invol} we see that involutive parabolic subgroups are of type $B_k \times (A_1)^\ell$, and the intersection of subgroups of such type is still of this type; in $D_4$, proper involutive parabolic subgroups are of the form $(A_1)^\ell$ for $\ell \leq 3$, hence also stable by intersection. In type $D_{2k}$ with $k \geq 3$, one can consider a standard parabolic subgroup $P_1$ of type $D_4$ and a simple reflection $s\in D_{2k} \backslash P_1$ commuting with every simple reflection in $P_1$ except one. Then $P_2:= s P_1 s$ is again an involutive parabolic subgroup, and $P_1 \cap P_2$ has type $A_3$, hence is not involutive. For instance, in $D_6$, labeling the Dynkin diagram as in Figure~\ref{fig:d6}, we set $P_1= \langle s_1, s_2, s_3, s_4 \rangle$, $P_2= s_5 P_1 s_5 = \langle s_1, s_2, s_3, s_5 s_4 s_5 \rangle$, and $P_1 \cap P_2= \langle s_1, s_2, s_3 \rangle$: indeed $P_1 \cap P_2$ is parabolic, of rank at most three since $P_1$ and $P_2$ have rank four and are distinct, and $\langle s_1, s_2, s_3 \rangle$ is a rank-three parabolic subgroup included in $P_1 \cap P_2$. Note that the argument also works for types $E_7$ and $E_8$. 

The case $A_1$ is trivial, and in type $I_2(2k)$ with $k\geq 2$, proper involutive parabolic subgroups are of type $A_1$, hence stability by intersection is clear (but in this case, every element $w$ of the group has reflection length at most $2$, hence $[1,w]_T$ is necessarily a lattice). 

The case $H_3$ is already covered by our previously obtained result~\cite[Theorem 2.2]{Gobet} since $\ell_T(w_0)=3$ in this case. Nevertheless, we can reobtain it easily since proper involutive parabolic subgroups in type $H_3$ are of type $(A_1)^\ell$ for $\ell \leq 2$, hence stability by intersection is also clear.   

For types $F_4$ and $H_4$, we can use an argument similar to the one used for $D_{2k}$, $k\geq 3$. Consider a standard parabolic subgroup $P_1$ of type $B_3$ (respectively $H_3$) in $F_4$ (respectively $H_4$). Let $s$ be the unique simple reflection which remains, i.e., not lying in $P_1$. Then $P_2:= s P_1 s$ has the same type as $P_1$, hence is involutive, but $P_1 \cap P_2$ has type $A_2$ (respectively $I_2(5)$), hence involutive parabolic subgroups are not stable by intersection.
\end{proof}

\begin{figure}[h!]
 \begin{center}
 	\begin{tikzpicture}[scale=1.2, thick]
 		
 		\node[circle, draw, fill=white, minimum size=8mm] (s3) at (0, 0) {$s_3$};
 		\node[circle, draw, fill=white, minimum size=8mm] (s4) at (1, 0) {$s_4$};
 		\node[circle, draw, fill=white, minimum size=8mm] (s5) at (2, 0) {$s_5$};
 		\node[circle, draw, fill=white, minimum size=8mm] (s6) at (3, 0) {$s_6$};
 		\node[circle, draw, fill=white, minimum size=8mm] (s1) at (-1, 1) {$s_1$};
 		\node[circle, draw, fill=white, minimum size=8mm] (s2) at (-1, -1) {$s_2$};
 		\draw (s3) -- (s4) -- (s5) -- (s6);
 		\draw (s3) -- (s1);
 		\draw (s3) -- (s2);
 		
 	\end{tikzpicture}
 \end{center}
 \caption{Labeling of the Dynkin diagram of type $D_6$.}
  \label{fig:d6}
 \end{figure}
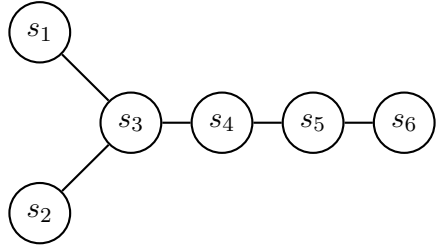
	
Combining Theorem~\ref{thm_main_lattice} with Lemma~\ref{dec_tech_pu}, we get: 	
	
\begin{cor}
Let $(W,S)$ be a Coxeter system and $u\in W$ an involution. Then $[1,u]_T$ is a lattice if and only if every irreducible factor of $P(u)$ is of one of the following types: $A_1, I_2(2k), k\geq 2, B_n, n\geq 3, D_4, H_3$.  
\end{cor}


\begin{thebibliography}{10}
		
		\bibitem{Arm} D.~Armstrong, \textsl{Generalized noncrossing partitions and combinatorics of Coxeter groups}, Mem. Amer. Math. Soc. {\bf 202} (2009), no. 949, x+159 pp.
		
	
		
		
		\bibitem{Bessis} D.~Bessis, \textsl{The dual braid monoid}, Ann.\
		Sci.\ \'Ecole Normale Sup\'erieure {\bf 36} (2003), 647--683.
		
		\bibitem{Bessis_free} D.~Bessis, \textsl{A dual braid monoid for the free group}, J. Algebra {\bf 302} (2006), 275-309.
		
			\bibitem{Biane} P.~Biane, \textsl{Some properties of crossings and partitions}, Discrete Math. {\bf 175} (1997), no. 1-3, 41-53.
			
			\bibitem{Bou} N~Bourbaki, \textsl{Éléments de mathématique. Fasc. XXXIV. Groupes et algèbres de Lie. Chapitre IV: Groupes de Coxeter et systèmes de Tits. Chapitre V: Groupes engendrés par des réflexions. Chapitre VI: systèmes de racines}, Actualités Scientifiques et Industrielles, 1337, Hermann, Paris, 1968. 
		
		\bibitem{BW_o} T.~Brady and C.~Watt, \textsl{A partial order on the orthogonal group},	Comm. Algebra {\bf 30} (2002), no. 8, 3749–3754.
		
		\bibitem{BW} T.~Brady and C.~Watt, \textsl{Non-crossing partition lattices in finite real reflection groups}, Trans. Amer. Math. Soc. {\bf 360} (2008), no. 4, 1983-2005.
		
		\bibitem{BH} B.~Brink and R.~Howlett, \textsl{A finiteness property and an automatic structure for Coxeter groups}, Math. Ann. {\bf 296} (1993), no. 1, 179–190.
		
				\bibitem{Carter} R.W.~Carter, \textsl{Conjugacy Classes in the Weyl
			group}, Compositio Math. {\bf 25} (1972), 1-59.
		
		
		\bibitem{DPS} E.~Delucchi, G.~Paolini, and M.~Salvetti, \textsl{Dual structures on Coxeter and Artin groups of rank three}, to appear in Geometry \& Topology (2023).

	\bibitem{Deodhar_invol} V.V.~Deodhar, \textsl{On the root system of a Coxeter group}, Comm. Algebra {\bf 10} (1982), no. 6, 611–630. 
		
	\bibitem{Deodhar} V.V.~Deodhar, \textsl{A note on subgroups generated by reflections in Coxeter groups},
		Arch. Math. (Basel) {\bf 53} (1989), no. 6, 543–546. 

	
		\bibitem{Dig1} F.~Digne, \textsl{Présentations duales des groupes de tresses de type affine $\tilde{A}$}, \textit{Dual presentations of braid groups of affine type $\tilde{A}$}, Comment. Math. Helv. {\bf 81} (2006), no. 1, 23-47.
		
		\bibitem{Dig2} F.~Digne, \textsl{A Garside presentation for Artin groups of type $\tilde{C}_n$}, Ann. Inst. Fourier {\bf 62} (2012), no. 2, 641-666.
		
			\bibitem{Dus} K.~Duszenko, \textsl{Reflection length in non-affine Coxeter groups}, Bull.
		Lond. Math. Soc. {\bf 44} (2012), no. 3, 571-577.
		
				\bibitem{Dyer_subgroups} M.J.~Dyer, \textsl{Reflection subgroups of Coxeter systems}, J. of Algebra {\bf 135} (1990), Issue 1, 57--73.
		
		\bibitem{Dyer_refn} M.J.~Dyer, \textsl{On minimal lengths of expressions of Coxeter group elements as products of reflections}, Proc. Amer. Math. Soc. {\bf 129} (2001), no. 9, 2591-2595.
		
		
		\bibitem{Gobet_dual} T.~Gobet, \textsl{Dual Garside structures and Coxeter sortable elements}, J. Comb. Algebra {\bf 4} (2020), no. 2, 167–213.
		
		\bibitem{Gobet} T. Gobet, \textsl{On maximal dihedral reflection subgroups and generalized noncrossing partitions}, Proc. Amer. Math. Soc. {\bf 152} (2024), no. 10, 4095–4101.
		
		\bibitem{Humphreys} J.E~Humphreys, \textsl{Reflection groups and Coxeter groups}, Cambridge Stud. Adv. Math., 29, Cambridge University Press, Cambridge, 1990, xii+204 pp.
		
		\bibitem{Kalli} M.~Kallipoliti, \textsl{The absolute order on the hyperoctahedral group}, J. Algebraic Combin. {\bf 34} (2011), no. 2, 183–211.
		
		\bibitem{BLMC} J.B.~Lewis, J.~McCammond, T.K.~Petersen, P.~Schwer, \textsl{Computing reflection length in an affine Coxeter group}, Trans. Amer. Math. Soc. {\bf 371} (2019), no. 6, 4097–4127.
		
		
		\bibitem{McCammond} J.~McCammond, \textsl{Dual euclidean Artin groups and the failure of the lattice property},
		J. of Algebra {\bf 437} (2015), 308-343.
		
		
		\bibitem{Paolini} G.~Paolini, \textsl{The dual approach to the $K(\pi, 1)$ conjecture}, Collection SMF Séminaires et Congrès {\bf 34}, 2025.
		
		\bibitem{Reading} N.~Reading, \textsl{Noncrossing partitions and the shard intersection order}, J. Algebraic Combin. {\bf 33} (2011), no. 4, 483–530.
		
		
		\bibitem{Serre} J.-P.~Serre, \textsl{Groupes de Coxeter finis: involutions et cubes}, Enseign. Math. {\bf 68} (2022), no. 1-2, 99–133.
		
		\bibitem{Sp} T.A.~Springer, \textsl{Some remarks on involutions in Coxeter groups}, Comm. Algebra {\bf 10} (1982), no. 6, 631–636. 
		
		\bibitem{Stanley} R.~Stanley, \textsl{Enumerative combinatorics. Volume 1}, Cambridge Stud. Adv. Math., 49
		Cambridge University Press, Cambridge, 2012, xiv+626 pp.
		
		\bibitem{STW} C.~Stump, H.~Thomas, and N.~Williams,
		\textsl{Cataland: why the Fuß?} Mem. Amer. Math. Soc. {\bf 305} (2025), no. 1535, vii+143 pp.
	
		
	\end{thebibliography}
\end{document}